\documentclass{article}

\usepackage{amsfonts,amsmath,amsthm,amssymb}
\usepackage{fullpage}
\usepackage{enumitem}
\usepackage{graphicx}
\usepackage[numbers]{natbib}
\usepackage{hyperref,xcolor,fullpage}

\usepackage{todonotes}

\usepackage[capitalise]{cleveref}
\crefname{equation}{}{}

\allowdisplaybreaks 
\usepackage[numbers]{natbib} 

\newtheorem{theorem}{Theorem}[section]
\newtheorem{lemma}[theorem]{Lemma}

\newtheorem{corollary}[theorem]{Corollary}

\newtheorem{proposition}[theorem]{Proposition}

\newcommand{\bi}{\textbf{i}}

\usepackage{mathrsfs}
\DeclareMathAlphabet{\mathpzc}{OT1}{pzc}{m}{it}
\newcommand{\fp}[1][]{ \ifthenelse{\isempty{#1}}{\mathpzc{p}}{\mathpzc{p}(#1)} }

\newcommand{\mA}{\mathcal{A}}

\newcommand{\CC}{\mathbb{C}}

\newcommand{\NN}{\mathbb{N}}

\newcommand{\bS}{\mathbb{S}}
\newcommand{\bD}{\mathbb{D}}

\newcommand{\rleq}{r^{\text{\tiny $\leq$}}}
\newcommand{\rle}{r^{\text{\tiny $<$}}}
\newcommand{\rstar}{r^{\star}}

\title{An Erd\H{o}s--Fuchs Theorem for Ordered Representation Functions}
\date{}
\author{
	Gonzalo Cao-Labora\thanks{ Universitat Polit\`ecnica de Catalunya, Facultat de Matem\`atiques i Estad\'istica, Pau Gargallo 5, 08028 Barcelona, Spain. E-mail: {\tt gonzalo.cao@estudiant.upc.edu}. This work was done under the research collaboration grant 2018 COLAB 00175 from AGAUR (Catalunya).}
	\and
	Juanjo Ru{\'e}\thanks{Universitat Polit\`ecnica de Catalunya and Barcelona Graduate School of Mathematics, Department of Mathematics, Edifici Omega, 08034 Barcelona, Spain. E-mail: {\tt juan.jose.rue@upc.edu}. Supported by the Spanish Ministerio de Econom\'{i}a y Competitividad project MTM2017-82166-P, and the Mar\'ia de Maetzu research grant MDM-2014-0445.}
	\and
	Christoph Spiegel\thanks{Universitat Polit\`ecnica de Catalunya, Department of Mathematics, Edificio Omega, 08034 Barcelona, Spain, and Barcelona Graduate School of Mathematics. E-mail: {\tt christoph.spiegel@upc.edu}. Supported by the Spanish Ministerio de Econom\'{i}a y Competitividad project MTM2017-82166-P and the Mar\'ia de Maetzu research grant MDM-2014-0445, and by an FPI grant under the project MTM2014-54745-P.}
}

\begin{document}
\maketitle

\begin{abstract}
	Let $k\geq 2$ be a positive integer. We study concentration results for the ordered representation functions $\rleq_k(\mA,n) = \# \big\{ (a_1   \leq   \dots   \leq   a_k) \in \mA^k : a_1+\dots+a_k = n \big\}$ and $
	 	\rle_k(\mA,n) = \# \big\{ (a_1 < \dots < a_k) \in \mA^k : a_1+\dots+a_k = n \big\}$ for any infinite set of non-negative integers $\mA$. Our main theorem is an Erd\H{o}s--Fuchs-type result for both functions: for any $c > 0$ and $\star \in \{\leq,<\}$ we show that
	\begin{equation*}
		\sum_{j = 0}^{n} \Big( \rstar_k (\mA,j) - c \Big)=  o\big(n^{1/4}\log^{-1/2}n\big)
	\end{equation*}
	is not possible. We also show that the mean squared error
	\begin{equation*}
		E^\star_{k,c}(\mA,n)=\frac{1}{n} \sum_{j = 0}^{n} \Big( \rstar_k(\mA,j) - c \Big)^2
	\end{equation*}
	satisfies $\limsup_{n \to \infty} E^\star_{k,c}(\mA,n)>0$. These results extend two theorems  for the non-ordered representation function proved by Erd\H{o}s and Fuchs in the case of $k=2$ (J. of the London Math. Society 1956).
\end{abstract}


\section{Introduction}
Let $\mA \subseteq \NN_0$ be a set of non-negative integers. For a given $k\geq 2$, we define its \emph{representation function} as
\begin{equation*}
	r_{k}(\mA,n) = \# \big\{ (a_1,\dots,a_k) \in \mA^k : a_1+\dots+a_k = n \big\}.
\end{equation*}
One of the most prominent conjectures in the theory of representation functions is the Erd\H{o}s--Tur\'an conjecture for additive bases~\cite{Erd-Turan41}: it states that if $r_{2}(\mA,n)>0$ for all $n$ large enough, then $\limsup_{n \to \infty} r_{2}(\mA,n)=\infty$ must hold. This conjecture remains unsolved, though some things are known about the behaviour of $r_{2}(\mA,n)$: a classical result of Erd\H{o}s and Fuchs, Theorem~1 in \cite{ErdFuc56}, shows that $\sum_{j = 0}^{n} \big( r_{2}(\mA,j) - c \big)=  o\big(n^{1/4}\log^{-1/2} n\big)$ cannot hold, that is $r_{2}(\mA,n)$ cannot be concentrated around a constant value. Jurkat (seemingly unpublished) and later Montgomery and Vaughan~\cite{MontVau90} improved upon the result of  Erd\H{o}s and Fuchs by showing that even $\sum_{j = 0}^{n} \big( r_{2}(\mA,j) - c \big)=  o(n^{1/4})$ cannot hold. On the other hand, Ruzsa~\cite{Ruzsa97} showed, through a probabilistic argument, that a sequence satisfying $\sum_{j = 0}^{n} \big( r_{2}(\mA,j) - c \big)=  O\big(n^{1/4} \log n\big)$ does exist.

There have been several generalisations of the result of Erd\H{o}s and Fuchs for the case of more than two summands, see~\cite{Hay81, RoSan13, Tang09}, and also for the case where the summands come from different sequences, see~\cite{DaiPan14, Hor04, Sark80}. In this paper we study a different variant: given $k \geq 2$ and some set $\mA \subseteq \NN_0$, we define its \emph{ordered representation functions} as
\begin{equation*}
	\rleq_k(\mA,n) = \# \big\{ (a_1,\dots,a_k) \in \mA^k : a_1   \leq   \dots   \leq   a_k, \:a_1+\dots+a_k = n \big\}
\end{equation*}
and
\begin{equation*}
	\rle_k(\mA,n) = \# \big\{ (a_1,\dots,a_k) \in \mA^k : a_1 < \dots < a_k, \:a_1+\dots+a_k = n \big\}.
\end{equation*}
Our first result shows that the ordered representation functions also satisfy an Erd\H{o}s--Fuchs-type result.
\begin{theorem} \label{thm:erdos_fuchs}
	Let $k \geq 2$, $c > 0$ and $\star \in \{\leq,<\}$. There does not exist any set $\mA \subseteq \NN_0$ satisfying
	\begin{equation*}\label{eq:Erd-Fuchs}
		\sum_{j = 0}^{n} \left( \rstar_k(\mA,j) - c \right)=  o\left(n^{1/4}\log^{-1/2}\!n\right).
	\end{equation*}
\end{theorem}
Our proof of this result differs from that of Theorem~1 in \cite{ErdFuc56} in two ways: while the proof of Erd\H{o}s and Fuchs is based on integrating an appropriate encoding of the problem along a small arc, our proof uses an	 integration with a smoothing function around the whole circle, which is reminiscent of the techniques used in~\cite{Sark80}. Additionally, in order to obtain an encoding of $\sum_{n = 0}^{\infty} \rstar_k(\mA,n) \, z^n$ in terms of the generating function $\sum_{a \in \mA} z^a$, we need to use the Symbolic Method. Through this encoding and the tools developed to prove \cref{thm:erdos_fuchs}, we are also able to prove a generalisation of another result of Erd\H{o}s and Fuchs, namely Theorem~2 in \cite{ErdFuc56}. For any $c \geq 0$ and $\mA \subseteq \NN_0$, let
\begin{equation*}
	E^\star_{k,c}(\mA,n) = \frac{1}{n}\sum_{j=0}^n \left( \rstar_k (\mA, j) - c \right)^2
\end{equation*}
denote the \emph{mean squared error}.
\begin{theorem} \label{thm:square_error}
	Let $k \geq 2$, $c \geq 0$, $\star \in \{\leq,<\}$ and $\mA = \{a_i \}_{i \in \NN} \subseteq \NN_0$ with $a_1 < a_2 < a_3 < \ldots$. If either $c > 0$ or the set $\{a_s/s^k\}_{s\in \NN}$ is bounded, then $\limsup_{n\to \infty} E^\star_{k,c}(\mA,n) > 0$.
\end{theorem}

Let us mention that a direct corollary of either \cref{thm:erdos_fuchs} or \cref{thm:square_error} is the following result.
\begin{corollary} \label{cor:non_constant}
	Let $k \geq 2$ and $\star \in \{\leq,<\}$. There does not exist any infinite set $\mA \subseteq \NN_0$ such that $\rstar_k(\mA,n)$ becomes constant for $n$ large enough.
\end{corollary}

For $k = 2$ this was already proved by Dirac~\cite{Dirac_1951}. To our knowledge, this seemingly easy result was not previously known for $k \geq 3$.

\paragraph{Outline.} In \cref{sec:preliminaries} we establish some preliminary results, first encoding the ordered representation functions through the generating functions of the given set, then establishing a dominant term under integration and finally stating some asymptotic bounds. In \cref{sec:P-T1} and \cref{sec:P-T2} we will respectively prove \cref{thm:erdos_fuchs} and \cref{thm:square_error}. Finally, in \cref{sec:furt-res} we discuss further research directions and open problems.


\section{Preliminaries} \label{sec:preliminaries}

Let us start by fixing some notation and giving an outline of the results we will establish in this section. Note that throughout this paper we let $\NN = \{1,2,3,\ldots\}$ and $\NN_0 = \NN \cup \{0\}$. We have already used the usual big $O$ and little $o$ notation in the introduction to state asymptotic results. For any two functions $f$ and $g$ we write $f = O(g)$ if there exists some $C > 0$ such that $f   \leq   Cg$ and $f = \Omega(g)$ if $g = O(f)$. We also write $f = o(g)$ if $f/g \to 0$ and $f = \omega(g)$ if $g = o(f)$. Here the arguments of the functions, and therefore the variable with respect to which the asymptotic statements are made, will either be $n \in \NN$ tending to infinity or $r \in (1/2,1)$ tending to $1$. Note that throughout the paper we assume that $r > 1/2$ simply to avoid any unimportant behaviour that occurs when $r$ is close to $0$. We will always specify with respect to which of the two we are making our asymptotic statements through an indexed $n$ or $r$.

In \cref{sec:encoding} we will find a way to express $\sum_{n = 0}^{\infty} \rstar_k(\mA,n) \, z^n$ through the generating function of the given set for $\star \in \{\leq,<\}$. Here $[z^k] f(z)$ will refer to the coefficient $c_k$ of a given formal power series $f(z) = \sum_{k \geq 0} c_k z^k$. In \cref{sec:tec-lemmas} we will establish a central statement that allows one to handle the case of $k > 2$. Here we will write
	\begin{equation*}
		\bS_r=\{z\in\CC: |z|=r\} \quad \text{and} \quad \bD=\{z\in\CC: |z|<1\}
	\end{equation*}
	for the circle of radius $r$, which as previously stated always lies in $(1/2,1)$ and will tend to $1$, and the open disc of radius $1$. The variable $z$ will always lie in $\bD$. Integrals along the circle $\bS_r$ will be taken with respect to the measure $d\mu = |dz|/(2\pi r) = d\theta/(2\pi)$ where $z = re^{i\theta}$. In particular, this implies that $\int_{\bS_r}d\mu =1$. We will also use the smoothing function
	\begin{equation*}
		h_M(z) = 1 + z + \ldots + z^{M-1} = \frac{ 1-z^M }{1-z}
	\end{equation*}
	where $M \in \NN$. Finally, \cref{sec:asymptotic_estimates} will contain some general asymptotic bounds that will be needed in the proofs of \cref{thm:erdos_fuchs} and \cref{thm:square_error}.

\subsection{Encoding the ordered representation functions}\label{sec:encoding}
The \emph{generating function} of a set $\mA \subseteq \NN_0$ is the formal power series
\begin{equation*}
	f_{\mA}(z) = \sum_{a \in \mA} z^a.
\end{equation*}
Observe that $f_{\mA}(z)$ is analytic in the open disc $\bD$ and that it is a strictly increasing in the real interval $(0,1)$. We use the terminology of Chapter~I in~\cite{FlajSedge} and define the \emph{size} of an element $a \in \mA$ to be $|a|=a$. For any $k$-tuple ${\bf a} = (a_1,\dots, a_k) \in \mA^k$ we also let its size be the sum of the sizes of the components, that is $|{\bf a}| = |a_1|+\dots+|a_k|$. Using this notation, $r(\mA,n)$ is simply the number of tuples in $\mA^k$ of size $n$. This interpretation directly gives
	\begin{equation*} \label{eq:r_genfun}
		\sum_{n = 0}^{\infty} r_k(\mA,n)\, z^n = f_{\mA}(z)^k,
	\end{equation*}
	which has been the basis of many proofs regarding the non-ordered counting function.

Expressing $\sum_{n = 0}^{\infty} \rstar_k(\mA,n)\, z^n$ in terms of $f_{\mA}(z)$ for $\star \in \{\leq,<\}$ is slightly more involved than this: for $\rleq_k(\mA,n)$ we need the \emph{multiset construction}, where one can use the same integer several times and the order of the summands does not matter, and for $\rle_k(\mA,n)$ we need the \emph{powerset construction}, where one cannot use the same integer several times and the order of the summands also does not matter.  See Section~I.2. in~\cite{FlajSedge} for precise definitions. Adapting Theorem~I.3. in~\cite{FlajSedge} to our setting therefore gives the expressions
	\begin{equation} \label{eq:rleq_genfun}
		\sum_{n = 0}^{\infty} \rleq_k(\mA,n) \, z^n= [u^k] \exp \left( \sum_{i=1}^{\infty} \frac{1}{i} u^i f_{\mA}(z^i)\right)
	\end{equation}
	as well as
	\begin{equation} \label{eq:rle_genfun}
		\sum_{n = 0}^{\infty} \rle_k(\mA,n) \, z^n= [u^k] \exp \left( \sum_{i=1}^{\infty} \frac{(-1)^{i+1}}{i} u^i f_{\mA}(z^i) \right).
	\end{equation}
	Writing
	\begin{equation*} \label{eq:Sk}
		S(k) = \big\{ \bi = (i_1,\dots,i_m) \in \NN^m : 1   \leq   m   \leq   k \text{ and } i_1 + \dots + i_m = k \big\}
	\end{equation*}
	and expanding the Taylor series of expressions \cref{eq:rleq_genfun} and \cref{eq:rle_genfun}, we get that
	\begin{equation} \label{eq:rleq_genfun_2}
		\sum_{n = 0}^{\infty} \rleq_k(\mA,n) \, z^n = \sum_{\bi \in S(k)} \frac{f_{\mA}(z^{i_1}) \cdots f_{\mA}(z^{i_{m}})}{i_1 \cdots i_{m} \cdot m!}
	\end{equation}
	as well as
	\begin{equation} \label{eq:rle_genfun_2}
		\sum_{n = 0}^{\infty} \rle_k(\mA,n) \, z^n  =  \sum_{\bi \in S(k)} (-1)^{m+k} \,  \frac{f_{\mA}(z^{i_1}) \cdots f_{\mA}(z^{i_{m}})}{i_1 \cdots i_{m} \cdot m!}.
	\end{equation}
	%
	%
Let us generalise our notation and write $\varepsilon_{\leq}(\bi) = 1/(i_1 \, \cdots \, i_{m} \cdot m!)$ as well as  $\varepsilon_{<}(\bi) = (-1)^{m+k} / (i_1 \, \cdots \, i_{m} \cdot m!)$ for any $\bi = (i_1,\dots, i_m)\in S(k)$, so that \cref{eq:rleq_genfun_2} and \cref{eq:rle_genfun_2} both become
	\begin{equation} \label{eq:rstar_genfun}
		\sum_{n = 0}^{\infty} \rstar_k(\mA,n) \, z^n = \sum_{\bi \in S(k)} \varepsilon_{\star}(\bi) \, f_{\mA}(z^{i_1}) \cdots f_{\mA}(z^{i_{m}})
	\end{equation}
	for $\star \in \{\leq,<\}$. Finally, in both of the proofs of \cref{thm:erdos_fuchs} and \cref{thm:square_error} we will argue that the term coming from the $k$-tuple ${\bf 1} = (1,\ldots,1) \in S(k)$ will asymptotically be dominant and therefore we let 
	\begin{equation*} \label{eq:S0k}
		S_0(k) = \big\{ \bi = (i_1,\dots,i_m) \in \NN^m : 1   \leq   m < k \text{ and } i_1 + \dots + i_m = k \big\}
	\end{equation*}
	denote the set of all remaining terms. Also note that $\varepsilon_{\leq}({\bf 1}) = \varepsilon_{<}({\bf 1}) = 1/k!$.

\subsection{The dominant term under integration} \label{sec:tec-lemmas}

The goal of this section is to formalise the previously mentioned fact that, when integrating over the right-hand side of \cref{eq:rstar_genfun}, the dominant term as $r$ tends to $1$ will come from the $k$-tuple ${\bf 1} \in S(k)$ whereas all terms coming from $S_0(k)$ will be negligible. We first prove the following lemma which establishes an application of Parseval's identity as well as H\"older's inequality that will be used throughout the rest of the paper. It is this technique that allows us to extend previously established Erd\H{o}s--Fuchs-type results to the case of $k > 2$.

\begin{lemma} \label{lemma:parseval_consequence}
	If $g(z) = \sum_{n=0}^{\infty} b_n z_n$ has non-negative integer coefficients and is analytic in $\bD$, then
	\begin{equation} \label{eq:parseval_consequence_1}
		\int_{\bS_r} |g(z)|^2 d\mu = \sum_{n=0}^{\infty} |b_n|^2 r^{2n}
	\end{equation}
	and if $k \geq 2$, then
	\begin{equation*} \label{eq:parseval_consequence_2}
		\int_{\bS_r} |g(z)|^k d\mu \geq g(r^2)^{k/2}.
	\end{equation*}
\end{lemma}

\begin{proof}
	For $k=2$ Parseval's identity gives us
	\begin{align*}
		\int_{\bS_r} |g(z)|^2 d\mu = \int_{\bS_r} g(z) \overline{g(z)} d\mu = \sum_{n,m\geq 0} b_n \overline{b_m} \int_{\bS_r} z^n \overline{z^m} d\mu = \sum_{n = 0}^{\infty} |b_n|^2 \, r^{2n} \geq \sum_{n = 0}^{\infty} b_n \big( r^2 \big)^n = g \big( r^2 \big).
	\end{align*}
	When $k > 2$, then H\"older's inequality and the observation for $k = 2$ establish that 
	\begin{equation*}
		\left( \int_{\bS_r} |g(z)|^k d\mu \right)^{2/k} \left( \int_{\bS_r} d\mu \right)^{(k-2)/k} \geq \int_{\bS_r} |g(z)|^2 d\mu \geq g \big( r^2 \big).
	\end{equation*}
	Noting that $\int_{\bS_r} d\mu = 1$ and raising the previous inequality to the $k/2$--th power gives us the result.
\end{proof}
We are now ready to prove the main statement of this section in three steps. We start with the following lemma.

\begin{lemma} \label{lem:lemita}
	For any $M,k,m \in \NN$ satisfying $m\leq k$ and any $\mA \subseteq \NN_0$ we have
	\begin{equation*}
	    \int_{\bS_r} |f_{\mA}(z)^k \, h_M(z)^2| d\mu \geq \Big( \int_{\bS_r} |f_{\mA}(z)^m \, h_M(z)^2| d\mu  \Big) \Big( \int_{\bS_r} |f_{\mA}(z)^2| d\mu  \Big)^{(k-m)/2}.
	\end{equation*}
\end{lemma}

\begin{proof}
	We start by considering three different cases concerning $k$ and $m$.

	\medskip
	\noindent\textbf{Case 1.} If $k=m+2$ and $k$ is even (and therefore $m$ is even), then let
		\begin{equation*}
			a_n^{M, m} = [z^n]f_{\mA}(z)^{m/2} \, h_M(z),
		\end{equation*}
		that is $f_{\mA}(z)^{m/2} \, h_M(z) = \sum_{n=0}^{\infty} a_n^{M, m} z^n$. Note that $f_{\mA}(z)^{k/2} \, h_M(z) = \big( f_{\mA}(z)^{m/2} \, h_M(z) \big) \, f_{\mA}(z)$ and therefore
		\begin{equation*}
			a_n^{M, k} = [z^n]\big( f_{\mA}(z)^{m/2} \, h_M(z) \big) \, f_{\mA}(z) = \sum_{i+j = n}[z^i] f_{\mA}(z)^{m/2} \, h_M(z) \, [z^j] f_{\mA}(z) =  \sum_{i+j=n} a_i^{M, m} \delta_j
		\end{equation*}
		for any $n \geq 0$. Here $\delta_j$ is the indicator function of $\mA$, that is $\delta_j = 1$ if $j \in \mA$ and $0$ otherwise. It follows that
	    \begin{equation*}
	    	\big| a_n^{M, k} \big|^2 \geq \sum_{i+j=n} \big|a_i^{M, m}\big|^2 \, |\delta_j|^2.
	    \end{equation*}
		Using this inequality as well as \cref{eq:parseval_consequence_1} in \cref{lemma:parseval_consequence}, we conclude that
		\begin{align*}
			\int_{\bS_r} |f_{\mA}(z)^k \, h_M(z)^2|d\mu  &= \sum_{n = 0}^{\infty} |a_n^{M, k}|^2 \, r^{2n}\geq \sum_{n = 0}^{\infty} \sum_{i+j=n} |a_i^{M, m}|^2 \, |\delta_j|^2 \, r^{2i+2j} \\
			& = \bigg( \sum_{i = 0}^{\infty} |a_i^{M, m}|^2 \, r^{2i}\bigg) \, \bigg( \sum_{j = 0}^{\infty} |\delta_j|^2 \, r^{2j}\bigg) \\
			&=  \int_{\bS_r} |f_{\mA}(z)^m \, h_M(z)^2|d\mu  \, \int_{\bS_r} |f_{\mA}(z)^{2}|d\mu  .
		\end{align*}

	\medskip
	\noindent\textbf{Case 2.} If $k = m+1$ and $k$ is even, then applying Cauchy--Schwarz and Case~1, we get that
		\begin{align*}
		    \int_{\bS_r} |f_{\mA}(z)^m \, h_M(z)^2|d\mu  &\leq \Big( \int_{\bS_r} |f_{\mA}(z)^{m+1} h_M(z)^2|d\mu  \Big)^{1/2} \Big( \int_{\bS_r} |f_{\mA}(z)^{m-1} h_M(z)^2|d\mu  \Big)^{1/2} \\
		    &\leq \int_{\bS_r} |f_{\mA}(z)^k \, h_M(z)^2|d\mu \, \Big( \int_{\bS_r} |f_{\mA}(z)^2|d\mu  \Big)^{-1/2}.
		\end{align*}
		Passing the last integral to the left-hand side establishes the statement in this case.
	
	\medskip
	\noindent\textbf{Case 3.} If $k=m+1$ and $k$ is odd, then applying Cauchy--Schwarz and Case~2, we get that
		\begin{align*}
		    \int_{\bS_r} |f_{\mA}(z)^m \, h_M(z)^2|d\mu &\leq \Big( \int_{\bS_r} |f_{\mA}(z)^{m+1} h_M(z)^2|d\mu  \Big)^{1/2} \Big( \int_{\bS_r} |f_{\mA}(z)^{m-1} h_M(z)^2|d\mu  \Big)^{1/2}\\
		    &   \leq   \Big( \int_{\bS_r} |f_{\mA}(z)^{k} h_M(z)^2|d\mu  \Big)^{1/2} \Big( \int_{\bS_r} |f_{\mA}(z)^m \, h_M(z)^2|d\mu  \Big)^{1/2} \Big( \int_{\bS_r} |f_{\mA}(z)|^2 d\mu \Big)^{-1/4}.
		\end{align*}
		Passing the last two integrals to the left-hand side and squaring establishes the statement in this case.

	\medskip

	Having established these three cases, the statement of the lemma now follows through an induction on $k$. Clearly the statement holds for $m = k = 1$. Assume now that it holds for $k-1$ and let us show that it then must also hold for $k$. The statement trivially holds for $m = k$ and Cases~2 and~3 establish that it also holds for $m = k-1$. For any $1 \leq m < k-1$ we can simply use the inductive assumption, since
	\begin{align*}
	    \int_{\bS_r} |f_{\mA}(z)^k \, h_M(z)^2| d\mu & \geq \Big( \int_{\bS_r} |f_{\mA}(z)^{k-1} \, h_M(z)^2| d\mu  \Big) \Big( \int_{\bS_r} |f_{\mA}(z)^2| d\mu  \Big)^{1/2} \\
	    & \geq \Big( \int_{\bS_r} |f_{\mA}(z)^{m} \, h_M(z)^2| d\mu  \Big) \Big( \int_{\bS_r} |f_{\mA}(z)^2| d\mu  \Big)^{(k-1-m)/2} \Big( \int_{\bS_r} |f_{\mA}(z)^2| d\mu  \Big)^{1/2} \\
	    & = \Big( \int_{\bS_r} |f_{\mA}(z)^{m} \, h_M(z)^2| d\mu  \Big) \Big( \int_{\bS_r} |f_{\mA}(z)^2| d\mu  \Big)^{(k-m)/2}.
	\end{align*}
	This proves the desired result.
\end{proof}

The following proposition is a slight generalisation of the previous lemma, allowing us to consider exponents in the arguments.
\begin{proposition} \label{prop:em_prop}
	For any $M,k,m,i \in \NN$ satisfying $i \mid M$ and $m   \leq   k$, and for any $\mA \subseteq \NN_0$, we have
	\begin{equation*} \label{eq:em_prop}
	    \int_{\bS_r} |f_{\mA}(z^i)^m \, h_M(z)^2|d\mu   \leq   \Big( \int_{\bS_r} |f_{\mA}(z)^k \, h_M(z)^2|d\mu \Big) \, i^2 \Big( \int_{\bS_r}  |f_{\mA}(z)|^2 d\mu \Big)^{-(k-m)/2}.
	\end{equation*}
\end{proposition}

\begin{proof}
	Let 
	\begin{equation*}
		a_n^{M, m} = [z^n]f_{\mA}(z)^{m/2}h_M(z),
	\end{equation*}
	that is $f_{\mA}(z)^{m/2}h_M(z) = \sum_{n=0}^{\infty} a_n^{M, m} z^n$. Since $M$ is a multiple of $i$, we can set $N = M/i$ and note that
	\begin{equation*}
		h_M(z) = h_N(z) \, h_i(z^N) = h_N(z) (1+z^N + \ldots + z^{(i-1)N}),
	\end{equation*}
	so that $f_{\mA}(z)^{m/2} \, h_M(z) = f_{\mA}(z)^{m/2} h_N(z)  \, (1 + z^N + \ldots + z^{(i-1)N} )$. In terms of coefficients, this implies that 
	\begin{equation*} \label{eq:coefficient_bound}
	    a_n^{M, m} = a_n^{N, m} + a_{n-N}^{N, m} + a_{n-2N}^{N, m} + \ldots +a_{n-(i-1)N}^{N, m} \geq a_n^{N, m},
	\end{equation*}
	where we let $a_n^{N, m} = 0$ for $n < 0$ and use the fact that all the coefficients are non-negative as both the coefficients of $f$ and $h_N$ are non-negative.
	
	Using that $h_M(z) = h_N(z) \, h_i(z^N)$, $|h_i(z)|   \leq   i$ as well as \cref{eq:parseval_consequence_1} in \cref{lemma:parseval_consequence} and \cref{eq:coefficient_bound}, we get
	\begin{align*}
		\int_{\bS_r} |f_{\mA}(z^i)^m \, h_M(z)^2|d\mu & = \int_{\bS_r} |f_{\mA}(z^i)^m h_N(z^i)^2| \, |h_i(z)^2|d\mu   \leq   i^2 \sum_{n = 0}^{\infty} |a_n^{N, m}|^2 \, r^{2in} \\
		&\leq i^2 \sum_{n = 0}^{\infty} |a_n^{M, m}|^2 \, r^{2n} = i^2 \int_{\bS_r} |f_{\mA}(z)^m \, h_M(z)^2| d\mu.
	\end{align*}
	We conclude the statement of the proposition by applying \cref{lem:lemita} to this inequality.
\end{proof}

The following corollary to the previous proposition establishes the main point of this section and is written in ready-to-use form for the proof of \cref{thm:erdos_fuchs}.
\begin{corollary} \label{corollary:of}
	For any $M,k,m,i \in \NN$ satisfying $i \mid M$ and $m < k$, and for any infinite set $\mA \subseteq \NN_0$, we have
    \begin{equation*}
    	 \int_{\bS_r} |f_{\mA}(z^i)^m \, h_M(z)^2|d\mu = o_r \! \left( \int_{\bS_r} |f_{\mA}(z)^k \, h_M(z)^2|d\mu \right).
    \end{equation*}
\end{corollary}

\begin{proof}
	By \cref{eq:parseval_consequence_1} in \cref{lemma:parseval_consequence} we have $\int_{\bS_r}|f_{\mA}(z)|^2 d\mu = \sum_{a\in \mA} r^{2a}$ which tends to infinity as $r$ tends to $1$ since $\mA$ is infinite. Therefore, the result directly follows from \cref{prop:em_prop}.
\end{proof}

Finally, for the proof of our \cref{thm:square_error} we will need an analogous result without the smoothing function. 

\begin{lemma} \label{lemma:integral_without_h} For any $m, k, i \in \NN$ satisfying $m < k$, and for any infinite set $\mA \subseteq \NN_0$, we have
\begin{equation*}
\int_{\bS_r} |f_{\mA}(z^i)^m |d\mu = o_r \! \left( \int_{\bS_r} |f_{\mA}(z)^k|d\mu \right).
\end{equation*}
\end{lemma}
\begin{proof}
	First assume that $m$ is even. Let $f_{\mA} (z )^{m/2} = \sum_{n=0}^\infty b_n z^{n}$ and note that $b_n \geq 0$ for all $n \in \NN_0$. Using \cref{eq:parseval_consequence_1} in \cref{lemma:parseval_consequence} and H\"older's inequality, we obtain
	\begin{equation*}
		\int_{\bS_r} |f_{\mA}(z^i)^m |d\mu = \sum_{n=0}^\infty  b_n^2  r^{2ni} \leq \sum_{n=0}^\infty b_n^2 r^{2n} = \int_{\bS_r} |f_{\mA}(z)^m |d\mu \leq \left( \int_{\bS_r} |f_{\mA}(z)^k |d\mu \right)^{m/k}.
	\end{equation*}
	Again by H\"older's inequality we know that $\int_{\bS_r}|f_{\mA}(z)|^k d\mu \geq \big( \int_{\bS_r}|f_{\mA}(z)|^2 d\mu \big)^{k/2} = \omega_r(1)$, so that the statement follows for even $m$ since $k < m$.
	
	Now assume that $m$ is odd. Let $f_{\mA }(z)^{(m+1)/2} (z) = \sum_{n=0}^\infty b_n' z^n$ and again note that $b_n' \geq 0$ for all $\NN_0$. Applying \cref{eq:parseval_consequence_1} in \cref{lemma:parseval_consequence}, we get
	\begin{equation} \label{eq:m+1_lastlemma}
		\int_{\bS_r} |f_{\mA}(z^i)^{m+1} |d\mu = \sum_{n=0}^\infty  b_n^{\prime 2}  r^{2ni} \leq \sum_{n=0}^\infty b_n^{\prime 2} r^{2n} = \int_{\bS_r} |f_{\mA}(z)^{m+1} |d\mu = O_r \! \left( \int_{\bS_r} |f_{\mA}(z)^k|d\mu \right).
	\end{equation}
	The last equality is trivial if $m+1=k$ and follows from the even case if $m+1 < k$. Using Cauchy-Schwarz, the even case for $m-1$ as well as \cref{eq:m+1_lastlemma}, we obtain
	\begin{align*}
		\int_{\bS_r} |f_{\mA}(z^i)^m |d\mu &\leq \left( \int_{\bS_r} |f_{\mA}(z^i)^{m-1} |d\mu \right)^{1/2} \left( \int_{\bS_r} |f_{\mA}(z^i)^{m+1} |d\mu \right)^{1/2} \\
		&= o_r \left( \int_{\bS_r} |f_{\mA}(z)^k|d\mu \right)^{1/2} O_r \left( \int_{\bS_r} |f_{\mA}(z)^k|d\mu \right)^{1/2}, 
	\end{align*} 
	proving the statement for $m$ odd.
\end{proof}

\subsection{Some asymptotic bounds} \label{sec:asymptotic_estimates}

Let us finally establish some general asymptotic bounds that will be needed in the proofs of \cref{thm:erdos_fuchs} and \cref{thm:square_error}.

\begin{lemma} \label{lemma:o(fi)}
	Let $i \in \NN$ and suppose that $g(z)=\sum_{n = 0}^{\infty} b_n z^n$ has non-negative coefficients and is analytic in $\bD$. If $\sum_{n=0}^\infty  b_n = \infty$, then
	\begin{equation*}
		g(r^i) = O_r  \big(g(r)^i\big).
	\end{equation*}
\end{lemma}

\begin{proof}
	Since $g(z)$ has non-negative coefficients, $g(r)$ is monotone increasing in $r$ and tends to infinity as $r$ tends to $1$. Let $(r_n)_{n\geq 1}$ be an arbitrary increasing sequence tending to $1$ in $(1/2,1)$. Since $r_n^i < r_n$, we have $g(r_n^i) < g(r_n) = o\big( g(r_n)^i \big)$ for all $n$ due to the monotonicity of $g$.
\end{proof}

\begin{lemma} \label{lemma:elliptic_integral}
	We have
	\begin{equation*}
		\int_{\bS_r} \frac{d\mu}{|1-z|} = O_r  \big( - \log (1-r) \big).
	\end{equation*}
\end{lemma}

\begin{proof}
	Let $z = r e^{i\theta}$. Using the fact that $|\sin \theta | \geq |\theta|/2$ for any  $\theta \in [-\pi/2, \pi/2]$, we can bound the integral as
	\begin{align*}
		\int_{\bS_r} \frac{d\mu}{|1-z|}  &= \frac{1}{2\pi r} \int_{-\pi /2}^{3\pi /2} \frac{d\theta}{\sqrt{(1-r\cos \theta)^2 + (r\sin \theta)^2}} \\
		&   \leq   \frac{1}{2 \pi r} \int_{-\pi /2}^{\pi /2} \frac{d\theta}{\sqrt{(1-r)^2 + (r \sin \theta)^2 }} + \frac{1}{2 \pi r} \int_{\pi/2}^{3\pi/2} \frac{d\theta}{\sqrt{1+r^2}} \\
		&   \leq   \frac{1}{2 \pi r} \int_{-\pi/2}^{\pi/2} \frac{d\theta}{\sqrt{(1-r)^2 + r^2 \theta^2 /4 }} + 1 
		\leq \frac{\sqrt{2}}{ \pi r} \int_{0}^{\pi/2} \frac{d\theta}{\big( 1-r + r\theta / 2 \big)} + 1 \\
		&   \leq   \frac{2 \sqrt{2} }{\pi r^2} \log \left( \frac{\pi/4}{1-r} + 1 \right) + 1   = O_r  \big( - \log (1-r) \big),
	\end{align*}
	where we have used the root-mean square--arithmetic mean inequality.
\end{proof}

\begin{lemma} \label{lemma:sum_with_logs}
	For any sequence of real numbers $e_n$ satisfying $e_n = o_n \big( n^{1/4} \log^{-1/2} n \big)$ we have 
	\begin{equation*} \label{eq:o_sum_logs_result}
		\sum_{n=0}^{\infty}  e_n^2 \, r^{2n} = o_r \! \left(  \frac{-1}{(1-r)^{3/2} \log \left( 1-r \right) } \right).
	\end{equation*}
\end{lemma}

\begin{proof}
	We start by showing that that there exists some $C_0  > 0$ such that
	\begin{equation} \label{eq:sum_with_logs}
		\sum_{n = 2}^{\infty} \big(r^{n} \, \log^{-1/2} (n) \, n^{1/4} \big)^2  \leq C_0 \, \frac{-1}{(1-r)^{3/2} \log \left( 1-r \right) }.
	\end{equation}
	Let
	\begin{equation*} \label{eq:N}
		N_0 = N_0(r) = \left\lfloor (1-r)^{-1/2} \right\rfloor = \omega_r(1).
	\end{equation*}
	Since $r < 1$ and $\log^{-1/2} (n)$ is decreasing, we can see that 
	\begin{align} \label{eq:sum_with_logs_first}
		\sum_{n=2}^{N_0} \big( r^{n} \, \log^{-1/2} (n) \, n^{1/4} \big)^2 &\leq \log^{-1} (2) \sum_{n=2}^{N_0} n^{1/2}   \leq   2 \, {N_0}^{3/2} \leq 2 \, (1-r)^{-3/4} \nonumber \\
		&= o_r \! \left( \frac{-1}{(1-r)^{3/2} \log \left( 1-r \right)} \right).
	\end{align}
	In order to deal with the remainder of the sum, we start by noting that for any $\delta > 1$ we have
	\begin{align*}
		(1-x)^{-\delta} = \sum_{n=0}^\infty c_n x^n & = \sum_{n=0}^\infty (-1)^n \frac{-\delta \, (-\delta - 1) \cdots (-\delta - n + 1)}{n!} \, x^n \\
		& = \sum_{n=0}^\infty \frac{\delta \, (\delta + 1) \cdots (\delta + n - 1)}{n!} \, x^n.
	\end{align*}
	The asymptotic behaviour of the logarithm of the coefficients of the Taylor expansion of $(1-x)^{-\delta}$ is therefore given by 
	\begin{align*}
		\log c_n & = \sum_{i=1}^{n} \log \left( \frac{\delta + i - 1}{i} \right) \geq \int_1^{n+1} \log \left( \frac{\delta+x-1}{x} \right) dx
	\end{align*}
	where we have used that that $1+ (\delta- 1)/i$ is decreasing in $i$ since $\delta > 1$. Using that
	\begin{equation*}
		\int \log \left( \frac{\delta-1}{x} + 1 \right) dx = (\delta-1) \log (\delta-1+x) + x \log \left( \frac{\delta-1}{x} + 1 \right)
	\end{equation*}
	and $x \log \big( 1+ (\delta-1)/x  \big) > 0$ for any $x > 0$, it follows that
	\begin{align*}
	\log c_n & \geq (\delta-1) \, \log n - (\delta-1) \log \delta  - \log \delta .
	\end{align*}
	We therefore have for some appropriate constant $C_1 = C_1(\delta)$ such that
	\begin{equation*}
		n^{\delta-1} \leq C_1 c_n
	\end{equation*}
	for any $n \geq 1$. Using this estimate with $\delta = 3/2$, we get 
	\begin{align} 
		\sum_{n = {N_0}+1}^{\infty} \big(r^{n} \, \log^{-1/2} (n) \, n^{1/4} \big)^2 &\leq   \log^{-1} (N_0) \, C_1 \, \sum_{n=1}^{\infty} c_n r^{2n} =  C_1 \, \frac{1}{(1-r^2)^{3/2} \log  (N_0)} \nonumber \\
		&= O_r \! \left( \frac{-1}{(1-r)^{3/2} \log (1-r)} \right).  \label{eq:sum_with_logs_second}
	\end{align}
	Collecting the estimate for $2   \leq   n   \leq   {N_0}$ in \cref{eq:sum_with_logs_first} and the estimate for $n \geq {N_0}+1$ in \cref{eq:sum_with_logs_second} establishes \cref{eq:sum_with_logs}.
		
	Now let $\varepsilon > 0$ and choose $N_1 \geq 2$ large enough such that
	\begin{equation*}
		e_n   \leq   \frac{ \varepsilon ^{1/2}}{ C_0^{1/2}} \, n^{1/4} \log^{-1/2} (n)
	\end{equation*}
	for all $n \geq N_1$. From \cref{eq:sum_with_logs} it follows that
	\begin{equation} \label{eq:o_sum_logs_1}
		\sum_{n = N}^{\infty} e_n^2 \, r^{2n}   \leq    \frac{-\varepsilon}{(1-r)^{3/2} \log (1-r)} .
	\end{equation}
	On the other hand,  there exists $r_0 = r_0(N_1)$ close enough to $1$ so that for any $r \geq r_0$
	\begin{equation} \label{eq:o_sum_logs_2}
		\sum_{n=0}^{N-1} e_n^2 r^{2n} \leq \sum_{n=0}^{N-1} e_n^2   \leq   \frac{-\varepsilon}{(1-r)^{3/2} \log (1-r)}
	\end{equation}
	since the right-hand side tends to $\infty$ as $r$ tends to $1$ while the left-hand side remains constant. Combining \cref{eq:o_sum_logs_1} and \cref{eq:o_sum_logs_2} and letting $\varepsilon$ tend to $0$ as $r$ tends to $1$, the statement of the lemma follows.
\end{proof}

\begin{lemma} \label{lemma:series_thm_2}
	For any $D > 0$ and $k \geq 1$, we have
	\begin{equation*}
		\sum_{s=1}^{\infty} r^{Ds^k} = \Omega_r \left( \left( 1-r \right)^{-1/k} \right).
	\end{equation*}
\end{lemma}
\begin{proof}
	Since $0 < r < 1$, we note that the function $f(x) = r^{Dx^k}$ is strictly decreasing in $[0, \infty)$. It follows that we may lower bound the series $\sum_{s=1}^{\infty} r^{Ds^k}$ by its corresponding integral, that is
	\begin{align*}
		\sum_{s=1}^{\infty} r^{ Ds^k} &\geq \int_{1}^\infty r^{Dx^k} dx = \int_{1}^\infty e^{-D|\log r| x^k} dx
		= \frac{\int_{1}^\infty e^{-y^k} dy}{D^{1/k}|\log r|^{1/k}}  = \Omega_r \left( \left( 1-r \right)^{-1/k} \right),
	\end{align*}
	where we have used that  $\int_{1}^\infty e^{-y^k} dy < \infty$.
\end{proof}


\section{Proof of \cref{thm:erdos_fuchs}} \label{sec:P-T1}

Throughout this proof $k \geq 2$, $c > 0$, $\star\in\{\leq,<\}$ and $\mA \subseteq \NN_0$ are fixed. Furthermore,  $z$ will always lie in $\bD$ and $r$ in the open interval $(1/2,1)$. We write
\begin{equation}\label{eq:beg}
	e_n = \sum_{j = 0}^{n} \Big( \rstar_k(\mA,j) - c \Big)
\end{equation}
for $n \geq 0$ and assume that, counter to the statement of  \cref{thm:erdos_fuchs}, we have $e_n = O_n\big(n^{1/4}\log^{-1/2} n\big)$. Multiplying \cref{eq:beg} by $z^n$ and summing over $n \geq 0$ gives us
\begin{equation*}
	\sum_{n = 0}^{\infty} e_n  z^n  +\frac{c}{(1-z)^2} = \sum_{n = 0}^{\infty} \sum_{j = 0}^{n} \rstar_k(\mA,j) \, z^n = \frac{1}{1-z}  \sum_{n = 0}^{\infty} \rstar_k(\mA,n) \, z^n
\end{equation*}
where we have used the fact that $\sum_{n = 0}^{\infty} (n+1) \, z^n = 1/(1-z)^2$. Applying~\cref{eq:rstar_genfun}, it follows that
\begin{equation} \label{eq:main}
	(1-z) \, \sum_{n = 0}^{\infty}  e_n \, z^n +  \frac{c}{1-z} = \sum_{\bi \in S(k)} \varepsilon_{\star}(\bi) \, f_{\mA}(z^{i_1}) \cdots f_{\mA}(z^{i_{m}}).
\end{equation}
Multiplying by the smoothing function $h_M(z)^2$, for some $M = M(r)$ to be determined later, taking absolute values and integrating along $\bS_r$, we obtain
\begin{equation} \label{eq:main_int}
	\int_{\bS_r} \big| (1-z)h_M(z)^2  \sum_{n = 0}^{\infty}  e_n  z^n \big| d\mu + \int_{\bS_r}  \Big| \frac{c \, h_M (z)^2}{1-z} \Big| d\mu  \geq  \int_{\bS_r} \big| h_M(z)^2 \sum_{S(k)} \varepsilon_{\star}(\bi) \, f_{\mA}(z^{i_1}) \cdots f_{\mA}(z^{i_{k}}) \big| d\mu.
\end{equation}

\paragraph{Bounding the left-hand side.} We note that $(1-z) \, h_M(z)^2 = (1-z^M) \, h_M(z)$ as well as $|h_{M}(z)| \leq M$ and $|1-z^M| \leq 2$. Applying Cauchy--Schwarz, we therefore get that {the left-hand side of \cref{eq:main_int} is at most 
\begin{equation*}
	2 \left( \int_{\bS_r} |h_M(z)|^2 d\mu \right)^{1/2} \left( \int_{\bS_r} \Big| \sum_{n = 0}^{\infty}  e_n \, z^n \Big|^2 d\mu \right)^{1/2} + \int_{\bS_r} \frac{cM^2}{|1-z|}d\mu .
\end{equation*}
Applying \cref{eq:parseval_consequence_1} in \cref{lemma:parseval_consequence} and then \cref{lemma:sum_with_logs} to the first term as well as \cref{lemma:elliptic_integral} to the second term, we can further bound the left-hand side of \cref{eq:main_int} by 
\begin{align} 
	& \quad  2 \left(  \sum_{n=0}^{M-1} r^{2n} \right)^{1/2} \left( \sum_{n = 0}^{\infty} |e_n|^2 \, r^{2n} \right)^{1/2} + M^2 \, O_r  \big( -\log  (1-r) \big) \nonumber \\
	 & \leq   M^{1/2} \: o_r \! \left( \frac{-1}{(1-r)^{3/4} \log^{1/2} (1-r) } \right) + M^2 \, O_r  \big( -\log  (1-r) \big).\label{eq:lhs_upperbound}
\end{align}

\paragraph{Bounding the right-hand side.}  Let us now observe that the right-hand side of~\cref{eq:main_int} is at least 
\begin{equation*}
	 \frac{1}{k!} \int_{\bS_r}  \big|  h_M(z)^2 f_{\mA}^k(z) \big| d\mu -  \sum_{\bi \in S_0(k)} \int_{\bS_r}  \big|  h_M(z)^2 \varepsilon_{\star}(\bi) \, f_{\mA}(z^{i_1}) \cdots f_{\mA}(z^{i_{m}}) \big| d\mu.
\end{equation*}
If $M$ is a multiple of $\textrm{lcm} \{ 1, 2, \ldots, k \}$, then we can use Cauchy--Schwarz and \cref{corollary:of} to upper bound the individual summands of the second term (where $m < k$) by
\begin{align*}
	|\varepsilon_{\star}(\bi)| \left( \int_{\bS_r} h_M (z)^2 f_{\mA} (z^{i_1})^m d\mu \right)^{1/m} \!\!\!\!\!\! \cdots \left( \int_{\bS_r} h_M (z)^2 f_{\mA} (z^{i_m})^m d\mu \right)^{1/m} = o_r \! \left( \int_{\bS_r} \big| h_M(z)^2 f_{\mA} (z)^k \big| d\mu\right)
\end{align*}
as $r$ tends to $1$. It follows that the term with $m = k$ is the dominant one and the terms coming from $S_0(k)$ are negligible. Using \cref{eq:parseval_consequence_2} in \cref{lemma:parseval_consequence}, we therefore know that the right-hand side of~\cref{eq:main_int} is at least 
\begin{align*}
	\left( \frac{1}{k!} + o_r(1)\right) \int_{\bS_r} \big| h_M(z)^2 f_{\mA}(z)^k \big| d\mu \geq \left( \frac{1}{k!} + o_r(1)\right) \, f_{\mA} \big( r^2 \big)^{k/2} h_M \big( r^2 \big).
\end{align*}
In order to estimate $f_{\mA}(r^2)^{k/2}$, we first note that by \cref{lemma:o(fi)} we have
\begin{equation*} 
	\sum_{\bi \in S(k)} \varepsilon_{\star}(\bi)f_{\mA}(r^{2i_1}) \cdots f_{\mA}(r^{2i_{m}}) = f_{\mA}(r^2)^k +  \sum_{\bi \in S_0(k)} O_r  \big( f_{\mA}(r^2)^m \big)  = (1 + o(1)) \, f_{\mA}(r^2)^k.
\end{equation*}
We secondly note, using \cref{lemma:sum_with_logs}, that
\begin{equation*}
 \sum_{n=0}^\infty e_n r^{2n} \leq \sum_{n=0}^\infty (1+e_n^2) r^{2n} \leq \frac{1}{1-r^2} + o_r \left( \frac{-1}{(1-r)^{3/2} \log (1-r) } \right).
\end{equation*}
Substituting $z=r^2$ in \cref{eq:main}, noting that $(1-r^2) = (2 + o_r(1)) \, (1-r)$ and using the two previous equations, it follows that 
\begin{equation*} 
	f_{\mA}(r^2)^k = (k! +o_r(1)) \left( \frac{c}{2(1-r)} + o_r \! \left( \frac{-1}{(1-r)^{1/2} \log (1-r)} \right) \right) = (k! + o_r(1)) \,\frac{c}{2(1-r)}.
\end{equation*}
Let us now choose
\begin{equation} \label{eq:M}
	M = k! \left\lceil \varepsilon \, \frac{-\log^{-1}(1-r)}{(1-r)^{1/2}} \right\rceil
\end{equation}
for some fixed $\varepsilon > 0$. This choice satisfies both $M(r) = \omega_r(1)$ and $r^{M(r)} = \Omega_r(1)$. Note also that $\text{lcm} \{ 1, 2, \ldots k \}$ divides $M$ as previously required. It follows that $h_M(r^2) \geq M/C_2$ for some $C_2 > 1$ and therefore the right-hand side of \cref{eq:main_int} is at least
\begin{equation} \label{eq:rhs_lowerbound}
	\left(  \frac{M\sqrt{c}}{\sqrt{2k!} C_2} + o_r(1) \right) \, (1-r)^{-1/2} = M \, \Omega_r \big( (1-r)^{-1/2} \big).
\end{equation}

\paragraph{Obtaining the contradiction.} Combining our bounds for the left- and right-hand sides of \cref{eq:main_int}, that is \cref{eq:lhs_upperbound} and \cref{eq:rhs_lowerbound}, we obtain
\begin{equation*} \label{eq:contra_final}
 	M \, \Omega_r \big( (1-r)^{-1/2} \big)   \leq   M^{1/2} \, o_r \! \left( \frac{-\log^{-1/2} (1-r) }{(1-r)^{3/4}} \right) + M^2 \, O_r  \big( - \log  (1-r) \big).
\end{equation*}
Inserting \cref{eq:M}, we therefore have
\begin{equation*}
	\varepsilon \, \Omega_r \! \left( \frac{-\log^{-1}(1-r)}{1-r}\right)    \leq  \varepsilon^{1/2} o_r \! \left( \frac{-\log^{-1} (1-r) }{1-r} \right) + \varepsilon^2  O_r \! \left( \frac{-\log^{-1}(1-r)}{1-r}\right),
\end{equation*}
with the constants in $\Omega_r$, $o_r$ and $O_r$ independent of $\varepsilon > 0$. Therefore, we have that $C_3 \varepsilon \leq o_r(1) + C_4 \varepsilon ^2$, for some $C_3, C_4 > 0$. This leads a contradiction taking any $\varepsilon < C_3/C_4$, so the assumption $e_n = o_n \big( n^{1/4} \log^{-1/2} (n) \big)$ was not possible. \hfill $\square$


\section{Proof of \cref{thm:square_error}} \label{sec:P-T2}

Throughout this proof $k \geq 2$, $c \geq 0$, $\star\in\{\leq,<\}$ and $\mA \subseteq \NN_0$ are fixed. Furthermore, $z$ will always lie in $\bD$ and $r$ in the open interval $(1/2,1)$. We start by noting that if $c$ is not an integer, then the statement immediately follows since 
\begin{equation*}
	\big( \rstar_k (\mA, n) - c \big)^2 \geq \max \{ \big( c-\big\lfloor c \big\rfloor \big)^2, \big( \big\lceil c \big\rceil  - c\big)^2 \} > 0.
\end{equation*}
We can therefore assume that $c \in \NN_0$. We now prove that we can assume that there exist $D > 0$ such that $a_s < D s^k$. If $c = 0$ this is given by the statement of the theorem, so let us consider $c \geq 1$. Using the fact that $c$ and $\rstar_k (\mA, n)$ are integers, we have
\begin{align} \label{eq:1_theorem_2}
	nE^\star_{k,c}(\mA,n) &= \sum_{j=0}^n \big( \rstar_k (\mA, j) - c \big)^2 \geq \sum_{j=0}^n | \rstar_k (\mA, j) - c | \geq  \Big| c(n+1) - \sum_{j=0}^n  \rstar_k (\mA, j) \Big|.
\end{align}
%
%
%
%
Since $a_{i_1} + a_{i_2} + \ldots + a_{i_k}   \leq   a_s$ trivially implies that every $i_j$ is at most $s$ for any $s,i_1,\ldots,i_k \in \NN$, it follows that $\sum_{j=0}^{a_s} \rstar_k (\mA, j)   \leq   s^k$. Taking $n = a_s$ in \cref{eq:1_theorem_2}, we therefore obtain
\begin{equation*}
	E^\star_{k,c}(\mA,a_s) \geq \frac{1}{a_s} \big( ca_s + c - s^k \big) = c + \frac{c}{a_s} - \frac{s^k}{a_s}.
\end{equation*}
Either the statement of the theorem holds, or $\limsup_{s \to \infty} E^\star_{k,c}(\mA,a_s) = 0$ implying that
\begin{equation*}
	\limsup_{s \to \infty} \left( c - \frac{s^k}{a_s} \right)   \leq   0.
\end{equation*}
It follows that we can assume $a_s   \leq   Ds^k$ for some appropriate $D > 0$. By \cref{eq:parseval_consequence_1} in \cref{lemma:parseval_consequence} as well as \cref{eq:rstar_genfun}, we now have
\begin{align*} \label{eq:cota_thm2_1}
	\left( \sum_{n = 0}^{\infty} \big( \rstar_k (\mA, n) - c\big)^2 \, r^{2n} \right)^{1/2} &= \left( \int_{\bS_r} \Big| \sum_{n = 0}^{\infty} \rstar_k (\mA, n) \, z^n - \frac{c}{1-z} \Big|^2 d\mu \right)^{1/2}\\
	& \geq \int_{\bS_r} \big|  \sum_{\bi \in S(k)} \varepsilon_{\ast}(\bi) f_{\mA}(z^{i_1}) \cdots f_{\mA}(z^{i_{m}}) - \frac{c}{1-z} \big| d\mu.
\end{align*}
Note that the terms with $\bi \in S_0(k)$ are negligible. Using Cauchy-Schwarz and \cref{lemma:integral_without_h}, we have that
\begin{align*}
\int_{\bS_r} \big|  \sum_{\bi \in S_0(k)} \varepsilon_{\ast}(\bi) f_{\mA}(z^{i_1}) \cdots f_{\mA}(z^{i_{m}}) \big| d\mu &\leq \sum_{\bi \in S_0(k)} \left( \int_{\bS_r} \big| f_{\mA}(z^{i_1}) \big|^m \right)^{1/m} \cdots \left( \int_{\bS_r} \big| f_{\mA}(z^{i_m}) \big|^m \right)^{1/m} \\
&= o_r \left( \int_{\bS_r} \big|  f_{\mA}(z)^k \big| d\mu \right).
\end{align*}
Now, \cref{eq:parseval_consequence_2} in \cref{lemma:parseval_consequence}, gives us
\begin{equation*}
\int_{\bS_r} \big|  \sum_{\bi \in S(k)} \varepsilon_{\ast}(\bi) f_{\mA}(z^{i_1}) \cdots f_{\mA}(z^{i_{m}}) \big| d\mu \geq \left( \frac{1}{k!} + o_r (1) \right) \int_{\bS_r} \big|  f_{\mA}(z)^k \big| d\mu  \geq \left( \frac{1}{k!} + o_r (1) \right) f_{\mA} (r^2 )^{k/2},
\end{equation*}
so that by \cref{lemma:elliptic_integral}
\begin{align*}
	\int_{\bS_r} \big|  \sum_{\bi \in S(k)} \varepsilon_{\ast}(\bi) f_{\mA}(z^{i_1}) \cdots f_{\mA}(z^{i_{m}}) - \frac{c}{1-z} \big| d\mu \geq \left( \frac{1}{k!} + o_r(1)\right) \, f_{\mA} \big( r^2 \big)^{k/2} - O_r  \big(- \log(1-r) \big).
\end{align*}
Now, taking into account that $a_s   \leq   Ds^k$ and using \cref{lemma:series_thm_2}, we have
\begin{align*}
	f_{\mA} (r^2)^{k/2} &= \left( \sum_{s = 1}^\infty r^{2a_s} \right)^{k/2} \geq \left( \sum_{s= 1}^\infty r^{2D s^k} \right)^{k/2} = \Omega_r \! \left( (1-r)^{-1/2} \right).
\end{align*}
Collecting all the bounds, it follows that
\begin{equation*}
	\sum_{n = 0}^{\infty} \big( \rstar_k (\mA, n) - c\big)^2 \, r^{2n} = \Omega_r \! \left( (1-r)^{-1/2} \right).
\end{equation*}
Therefore,
\begin{equation*}
	\sum_{n = 0}^{\infty} nE^\star_{k,c}(\mA,n) r^{2n} = \frac{1}{1-r^2} \sum_{n = 0}^{\infty} \big( \rstar_k(\mA, n) - c\big)^2 \, r^{2n} = \frac{1}{1-r^2} \, \Omega_r \left( \frac{1}{1-r^2} \right) \geq C_5 \, \sum_{n = 0}^{\infty} nr^{2n}
\end{equation*}
for some appropriate constant $C_5 > 0$. It follows that infinitely many of the coefficients $ nE^\star_{k,c}(\mA,n)$ must be greater than $ C_5n/2$ and hence $\limsup_{n \to \infty} E^\star_{k,c}(\mA,n) \geq C_5/2 > 0$ as desired. \hfill $\square$

\section{Further research}\label{sec:furt-res}

With \cref{thm:erdos_fuchs} we have established an Erd\H{o}s--Fuchs-type result for ordered representation functions, showing that an error term of the form $o(n^{1/4}\log^{-1/2} n)$ is not possible. It would be of interest to adapt the techniques in \cite{MontVau90}, see also \cite{Hay81}, in order to rule out an error term of the form $o(n^{1/4})$. However, the fact that one has to introduce several extra terms in the encoding of $r^{\star}_k (\mA,n)$ complicates this approach. 

In a different direction, the work of \cite{RoSan13}, building on a previous construction by \cite{Ruzsa97}, see also \cite{DaiPan14}, shows that for any $k \geq 2$ there exists a set $\mA_k$ and a constant $c>0$ such that $\sum_{j=0}^{n} (r_{k}(\mA_k,n)-c)=O(n^{1-3/(2k)})$, leaving a gap for $k\geq 3$. It would be very interesting to try to improve this bound for $k\geq 3$, both in the ordered and in the unordered case.

Finally, let us mention another challenging problem going back to S\'ark\"ozy and S\'os~\cite{SarkozySos_1997}. For given coefficients $k_1,\dots, k_l$ and a set $\mA \subseteq \NN_0$ , we consider the representation function 
\begin{equation*}
	r_{k_1,\dots,k_l}(\mA,n) = \# \big\{ (a_1,\dots,a_l) \in \mA^l :  k_1a_1+\dots+k_la_l = n\big\}.
\end{equation*}
In the case of $l=2$ the behaviour of $r_{k_1,k_2}(\mA,n)$ already drastically depends on the choice of the coefficients: for $(k_1,k_2)=(1,k)$ where $k>1$ Moser~\cite{Moser_1962} built an explicit set $\mA_0$ such that $r_{(1,k)}(\mA_0,n)=1$ for all value of $n$. However, for $(k_1,k_2)$ with $\gcd(k_1,k_2)=1$ and $1<k_1<k_2$,  Cilleruelo and Ru\'e shown in \cite{CillerueloRue_2009} that the corresponding representation function cannot become constant for $n$ large enough. See also \cite{RueSpiegel20} for extensions to $l$-fold sums. Consequently, obtaining an Erd\H{o}s--Fuchs-type result in this setting would require additional conditions, that is the arithmetic properties of the coefficients should play an important role in the arguments. See also~\cite{Rue_2011}, where some Erd\H{o}s--Fuchs-type results were obtained in this setting for certain types of coefficients.


\end{document}